\documentclass[reqno,12pt]{amsart}
\usepackage{amsmath,amssymb,amsthm,graphicx,mathrsfs,url}
\usepackage{wrapfig}
\usepackage{enumitem}
\usepackage{mathtools}
 \usepackage[T1]{fontenc}
\usepackage[usenames,dvipsnames]{color}
\usepackage[colorlinks=true,linkcolor=Red,citecolor=Green]{hyperref}
\usepackage[super]{nth}
\usepackage[open, openlevel=2, depth=3, atend]{bookmark}
\hypersetup{pdfstartview=XYZ}
\usepackage[font=footnotesize]{caption}
\usepackage{a4wide}

\usepackage{epstopdf}
 
\usepackage{hyperref}

\theoremstyle{plain}
\newtheorem{theo}{Theorem}
\newtheorem{lemm}{Lemma}[section]
\newtheorem{prop}[lemm]{Proposition}
\newtheorem{coro}[lemm]{Corollary}
\newtheorem{conj}{Conjecture}

\theoremstyle{definition}

\theoremstyle{remark}
\newtheorem{rema}{Remark}[section]
 
\numberwithin{equation}{section}

\newcommand{\R}{\mathbb{R}}

\newcommand{\N}{\mathbb{N}}

\newcommand{\ke}{\text{ker }}

\newcommand{\mc}{\mathcal}
\newcommand{\rr}{\mathbb{R}}
\newcommand{\nn}{\mathbb{N}}

\newcommand{\la}{\lambda}
\newcommand{\eps}{\epsilon}

\newcommand{\x}{\times}

\newcommand{\til}{\widetilde}

\newcommand{\cjd}{\rangle}
\newcommand{\cjg}{\langle}

\newcommand{\demi}{\tfrac{1}{2}}

\newcommand{\be}{\begin{equation}}
\newcommand{\ee}{\end{equation}}

\title[The marked length spectrum of Anosov manifolds]{The marked length spectrum of Anosov manifolds}

\author{Colin Guillarmou}

\address{Laboratoire de Math\'ematiques d'Orsay, Univ. Paris-Sud, CNRS, Universit\'e Paris-Saclay, 91405 Orsay, France}
\email{colin.guillarmou@math.u-psud.fr}

\author{Thibault Lefeuvre}
\address{Laboratoire de Math\'ematiques d'Orsay, Univ. Paris-Sud, CNRS, Universit\'e Paris-Saclay, 91405 Orsay, France}
\email{thibault.lefeuvre@u-psud.fr}

\begin{document}

\begin{abstract}
In all dimensions, we prove that the marked length spectrum of a Riemannian manifold $(M,g)$ with Anosov geodesic flow and non-positive curvature locally determines the metric in the sense that two close enough metrics with the same marked length spectrum are isometric. In addition, we provide a completely new stability estimate quantifying how the marked length spectrum controls the distance between the metrics. In dimension $2$ we obtain similar results for general metrics with Anosov geodesic flows.  
We also solve locally a rigidity conjecture of Croke relating volume and marked length spectrum for the same category of metrics. Finally, by a compactness argument, we show that the set of negatively curved metrics (up to isometry) with the same marked length spectrum and with curvature in a bounded set of $C^\infty$ is finite.
\end{abstract}

\maketitle

\section{Introduction}

Let $(M,g)$ be a smooth closed Riemannian manifold. If the metric $g$ admits an Anosov geodesic flow, the set of lengths of closed geodesics is discrete and is called the \emph{length spectrum of $g$}. 
It is an old problem in Riemannian geometry to understand if the length spectrum determines the metric $g$ up to isometry. Vigneras  \cite{mfv} found counterexamples in constant negative curvature.
On the other hand we know that the closed geodesics are parametrised by the set $\mc{C}$ of free-homotopy classes, or equivalently the set of conjugacy classes of the fundamental group $\pi_1(M)$. 
Indeed, for each $c\in \mc{C}$, there is a unique closed geodesic $\gamma_c$ of $g$ in the class $c$.  Particular examples of manifolds with Anosov geodesic flow are negatively curved compact manifolds.
We can thus define a map, called the \emph{marked length spectrum} 
\begin{equation}
\label{defLg}
L_g: \mc{C}\to \rr^+, \quad L_g(c):=\ell_g(\gamma_c)
\end{equation}
where, if $\gamma$ is a $C^1$-curve, $\ell_g(\gamma)$ denotes its length with respect to $g$.

We recall the following long-standing conjecture stated in Burns-Katok \cite{bk} (and probably considered even before):

\begin{conj}{\cite[Problem 3.1]{bk}}
\label{conj1}
If $(M,g)$ and $(M,g_0)$ are two closed manifolds with negative sectional curvature and same marked length spectrum, i.e $L_g=L_{g_0}$, then they are isometric, i.e. there exists a smooth diffeomorphism $\phi : M \rightarrow M$ such that $\phi^*g = g_0$.
\end{conj}
Note that if $\phi : M \rightarrow M$ is a diffeomorphism isotopic to the identity, then $L_{\phi^*g_0}=L_{g_0}$.
The analysis of the linearised operator at a given metric $g_0$ is now well-understood, starting from the pionnering work of Guillemin-Kazhdan \cite{GuKa}, and pursued by the works of Croke-Sharafutdinov \cite{cs}, Dairbekov-Sharafutdinov \cite{DaSh} and more recently by
Paternain-Salo-Uhlmann \cite{psu1,psu2} and the first author \cite{gu}. It is known that the linearised operator, the so called \emph{$X$-ray transform}, is injective for non-positively curved manifolds with Anosov geodesic flows in all dimensions, and for all Anosov geodesic flows in dimension $2$. These works imply the deformation rigidity result: there is no $1$-parameter family of such metrics with the same marked length spectrum.

Concerning the non-linear problem (Conjecture \ref{conj1}), there are only very few results:
in dimension $2$ and non-positive curvature, the breakthrough was due to Otal \cite{ot} and Croke \cite{cr1} who solved that problem\footnote{Otal's work was in negative curvature and Croke in non-positive curvature}. It was extended by Croke-Fathi-Feldman \cite{cff} to surfaces when one of the metrics has non positive curvature and the other has no conjugate points.
Katok \cite{Ka} previously had a short proof for metrics in a fixed conformal class, in dimension $2$ but that can be easily extended to higher dimensions. 
Beside the conformal case, for higher dimension the only known rigidity result is that of Hamenst\"adt  \cite{Ha} by applying the famous entropy rigidity work of Besson-Courtois-Gallot \cite{BCG}:  if two negatively curved metrics $(M,g)$ and $(M,g_0)$ have the same marked length spectrum and their Anosov foliation is $C^1$, then ${\rm Vol}_g(M)={\rm Vol}_{g_0}(M)$, and since  $L_g$ determines the topological entropy, the result of \cite{BCG} solves Conjecture \ref{conj1} when $g_0$ is a locally symmetric space. For general metrics the problem is largely open. We notice that Biswas \cite{Bi} recently announced that two Riemannian manifolds with same marked length spectrum are bi-Lipshitz homeomorphic.
We refer to the surveys/lectures of Croke and Wilkinson \cite{cr1,aw} for an overview of the subject.
The main difficulty to obtain a local rigidity result is that the linearised operator takes values on functions on a discrete set and is not a tractable operator to obtain non-linear results. 

The Conjecture \ref{conj1} actually also makes sense for Anosov geodesic flows, without the negative curvature assumption, but it might be more reasonable to conjecture that only finitely many non isometric Anosov metrics have same marked length spectrum.

Our first result is a local rigidity statement that says that the marked length spectrum parametrizes locally the isometry classes of metrics. As far as we know, this is the first (non-linear) progress towards Conjecture \ref{conj1} in dimension $n\geq 3$ for general metrics.
\begin{theo}
\label{th1}
Let $(M,g_0)$ be:
\begin{itemize}
\item either a closed smooth Riemannian surface with Anosov geodesic flow,
\item or a closed smooth Riemannian manifold of dimension $n \geq 3$ with Anosov geodesic flow and non-positive sectional curvature,
\end{itemize}
and let $N>3n/2+8$. There exists $\eps > 0$ such that for any smooth metric $g$ with same marked length spectrum as $g_0$ and such that $\|g-g_0\|_{C^{N}(M)} < \eps$, there exists a  diffeomorphism $\phi : M \rightarrow M$ such that $\phi^*g = g_0$.
\end{theo}
We actually prove a slightly stronger result in the sense that $g$ can be chosen to be in the H\"older space $C^{N,\alpha}$ with $(N,\alpha)\in \nn\x(0,1)$ satisfying $N+\alpha>3n/2+8$. Note also that $\eps > 0$ is chosen small enough so that the metrics $g$ have Anosov geodesic flow too. 
This result is new even if $\dim(M)=2$, as we make no assumption on the curvature. If 
$\dim(M)>2$ and $g$ is Anosov, the same result holds outside a finite dimensional manifold of metrics, see Remark \ref{kerI2An}.
This implies a general result supporting Conjecture \ref{conj1}: 
\begin{coro}
Let $(M,g_0)$ be an $n$-dimensional compact Riemannian manifold with negative curvature 
and let $N>3n/2+8$. Then there exists $\eps > 0$ such that for any smooth metric $g$ with same marked length spectrum as $g_0$ and such that $\|g-g_0\|_{C^{N}(M)} < \eps$, there exists a  diffeomorphism $\phi : M \rightarrow M$ such that $\phi^*g = g_0$.
\end{coro}
Since two $C^0$-conjugate Anosov geodesic flows that are close enough have the same marked length spectrum, we also deduce that for $g_0$ fixed as above, each metric $g$ which is close enough to $g_0$ and has geodesic flow conjugate to that of $g_0$ is isometric to $g$.

To prove these results, a natural strategy would be to apply an implicit function theorem. The linearised operator $I_{2}$, called the \emph{$X$-ray transform}, consists in integrating $2$-tensors along closed geodesics of $g_0$ (see Section \ref{Xray}). It is known to be injective under the assumptions of Theorem \ref{th1} 
by \cite{cs,psu1,psu2,gu}, but as mentioned before, the main difficulty to apply this to the non-linear problem is that $I_2$ maps to functions on the discrete set $\mc{C}$ and it seems unlikely that its range is closed. To circumvent this problem, we use some completely new approach from \cite{gu} that replaces the operator $I_2$ by a more tractable Fredholm one, that is constructed using microlocal methods in Faure-Sj\"ostrand \cite{fs} and Dyatlov-Zworski \cite{dz}. This new operator, denoted by $\Pi_2$, plays the same role as the normal operator $I_2^*I_2$ that is strongly used in the context of manifolds with boundary but $\Pi_2$ is not constructed from $I_2$. 
The additional crucial ingredient that allows us to relate the operators $I_2$ and $\Pi_2$ 
is a ``positive Livsic theorem'' due to Pollicott-Sharp \cite{ps} and Lopes-Thieullen \cite{lt}. We manage to obtain a sort of stability estimate for the $X$-ray operator with some loss of derivatives, but that is sufficient for our purpose.
A corollary of this method is a completely new stability estimate for the $X$-ray transform on divergence-free tensors, that quantifies the smallness of a divergence-free symmetric $m$-tensor $f\in C^\alpha(M;S^mT^*M)$ (for $m\in\nn$,$\alpha>0$) in terms of the supremum of its integrals  
$\frac{1}{\ell(\gamma)}\int_{\gamma}f$ over all closed geodesics $\gamma$ of $g_0$, see Theorem \ref{th:stab1}.\\ 

Combining these methods with some ideas developed by Croke-Dairbekov-Sharafutdinov 
\cite{cds} and the second author \cite{tl} in the case with boundary, we are able to prove a new rigidity result which has similarities with the minimal filling volume problem appearing for manifolds with boundary and is a problem asked by Croke in \cite[Question 6.8]{Cr2}. 
\begin{theo}
\label{th1bis}
Let $(M,g_0)$ be as in Theorem \ref{th1} and let $N>\frac{n}{2}+2$. There exists $\eps > 0$ such that for any smooth metric $g$ satisfying $\|g-g_0\|_{C^N} < \eps$, the following holds true: if $L_g(c)\geq L_{g_0}(c)$ for all conjugacy class $c\in\mc{C}$ of $\pi_1(M)$, 
then ${\rm Vol}_{g}(M)\geq {\rm Vol}_{g_0}(M)$. If in addition
${\rm Vol}_{g}(M)={\rm Vol}_{g_0}(M)$, then there exists a diffeomorphism $\phi : M \rightarrow M$ such that $\phi^*g = g_0$.
\end{theo}
Again, in the proof, we actually just need $g\in C^{N,\alpha}$ with $(N,\alpha)\in\nn\x(0,1)$ 
satisfying $N+\alpha>n/2+2$. This result (but without the assumption that $g$ is close to $g_0$) was proved by Croke-Dairbekov \cite{CrDa} for negatively curved surfaces and for metrics in a conformal class in higher dimension (by applying the method of \cite{Ka}). Our result 
is the first general one in dimension $n>2$ and is new even when $n=2$ as we do not assume negative curvature.\\

Next, we get H\"older stability estimates quantifying how close are metrics with close marked length spectrum. In that aim we fix a metric $g_0$ with Anosov geodesic flow and 
define for $g$ close to $g_0$ in some $C^{N}(M)$ norm
\[ \mc{L}(g)\in \ell^\infty(\mc{C}) ,\quad \mc{L}(g)=\frac{L_g}{L_{g_0}}.\]
We are able to show (here and below, $H^s(M)$ is the usual $L^2$-based Sobolev space of order $s\in \rr$ on $M$):
\begin{theo}
\label{th3}
Let $(M,g_0)$ satisfy the assumptions of Theorem \ref{th1} and let 
$N>3n/2+9$. For all $s>0$ small there is a positive $\nu =\mc{O}(s)$ and a constant $C > 0$ such that the following holds:  for all $\delta>0$ small, there exists $\eps>0$ small such that for any $C^N$ metric $g$ satisfying 
$\|g-g_0\|_{C^N}<\eps$, there is a diffeomorphism $\phi$ such that (here ${\bf 1}(c)=1$ for each $c\in \mc{C}$):
\[\|\phi^*g-g_0\|_{H^{-1-s}(M)} \leq C\delta \|\mc{L}(g)-{\bf 1}\|^{(1-\nu)/2}_{\ell^\infty(\mc{C})}+ 
C\|\mc{L}(g)-{\bf 1}\|_{\ell^\infty(\mc{C})} \]
\end{theo}
We note that this is the first quantitative estimate on the marked length rigidity problem. It is even new for negatively curved surfaces where the injectivity of $g\mapsto L_g$ (modulo isometry) is known by \cite{cr1,ot}. \\

We conclude by some finiteness results. On a closed manifold $M$, we consider for $\nu_1\geq \nu_0>0$, $\theta_0>0$ and $C_0>0$ the set of  smooth metrics $g$ with Anosov geodesic flow satisfying the estimates \eqref{defAnosov} where the constants $C,\nu$ verify $C\leq C_0$, 
$\nu\in [\nu_0,\nu_1]$ and $d_{G}(E_s,E_u)\geq \theta_0$ if $d_G$ denotes the distance in the Grassmanian of the unit tangent bundle $SM$ induced by the Sasaki metric. We write $\mc{A}(\nu_0,\nu_1,C_0,\theta_0)$ for the set of such metrics. This is a closed set that consists of uniform Anosov geodesic flows.
For example, metrics with curvatures contained in $[-a^2,-b^2]$ with $a>b>0$ satisfy such property \cite[Theorem 3.2.17]{Kl2}.
In what follows, we denote by $\mc{R}_g$ the curvature tensor of $g$.

\begin{theo}\label{finiteness}
Let $M$ be a smooth closed manifold and let $\nu_1\geq \nu_0>0$, $C_0>0$ and $\theta_0>0$.  
If $\dim M=2$, for each sequence of positive numbers $B:=(B_k)_{k\in\nn}$, 
there is at most finitely many isometry classes of metrics $g$ in 
$\mc{A}(\nu_0,\nu_1,C_0,\theta_0)$ satisfying the curvature bounds 
$|\nabla_g^k\mc{R}_{g}|_{g}\leq B_k$ and with the same marked length spectrum. If $\dim M>2$ the same holds true if in addition $g$ have non-positive curvature.
\end{theo}

Restricting to negatively curved metrics we get the finiteness results (new if $\dim M>2$): 
\begin{coro}
Let $M$ be a compact manifold. Then, for each $a>0$ and each sequence 
$B=(B_k)_{k\in\nn}$ of positive numbers, there is at most finitely many smooth  isometry classes of metrics with sectional curvature bounded above by $-a^2<0$, curvature tensor bounded by $B$ 
(in the sense of Theorem \ref{finiteness}) and same marked length spectrum.
\end{coro}

We remark that the $C^\infty$ assumptions on the background metric $g_0$ in all our results and the boundedness assumptions on the $C^\infty$ norms of the curvatures in Theorem \ref{finiteness} can be relaxed to $C^k$ for some fixed $k$ depending on the dimension.\footnote{The smoothness assumptions come from the fact we are using certain results based on microlocal analysis; it is a standard fact that only finitely many derivatives are sufficient for microlocal methods. It is likely that with some technical works one could improve the result to $C^3$ or $C^4$ regularity.}\\

\textbf{Acknowledgements.} We would like to thank K. Burns, S. Gou\"ezel, G. Knieper,  G. Massuyeau, M. Mazzucchelli, D. Monclair, F. Naud, F. Paulin, G. Rivi\`ere, M. Salo, J-M. Schlenker, R. Tessera for useful discussions. C.G. is supported partially by ERC consolidator grant IPFLOW.

\section{The marked length spectrum and its linearisation}

\subsection{Marked length spectrum}
We consider a smooth manifold $M$ equipped with a smooth Riemannian metric $g$. We let $\pi_1(M)$ be the fundamental group of $M$ and $\mc{C}$ be the set of conjugacy classes in $\pi_1(M)$. It is well-known that $\mc{C}$ corresponds to the set of free-homotopy classes of $M$.
Assume now that the geodesic flow $\varphi_t$ of $g$ on the unit tangent bundle $SM$ is Anosov. We will call \emph{Anosov manifolds} such Riemannian manifolds and let
\[\mc{A}:=\{ g\in C^\infty(M;S^2_+T^*M); g \textrm{ has Anosov geodesic flow}\}. \]
We recall that $\varphi_t$ with generating vector field $X$ is called Anosov if there exists some constants $C > 0$ and $\nu > 0$ such that for all $z = (x,v) \in SM$, there is a continuous flow-invariant splitting
\be \label{eq:split} T_z(SM) = \R X(z) \oplus E_u(z) \oplus E_s(z), \ee
where $E_s(z)$ (resp. $E_u(z)$) is the \textit{stable} (resp. \textit{unstable}) vector space in $z$, which satisfy
\be \label{defAnosov}
\begin{array}{c} 
|d\varphi_t(z).\xi|_{\varphi_t(z)} \leq C e^{-\nu t} |\xi|_{z}, ~~ \forall t \geq 0, \xi \in E_s(z) \\
|d\varphi_t(z).\xi|_{\varphi_t(z)} \leq C e^{-\nu |t|} |\xi|_{z}, ~~ \forall t \leq 0, \xi \in E_u(z)\end{array} \ee
The norm, here, is given in terms of the Sasaki metric of $g$. By Anosov structural stability \cite{An,DMM}, $\mc{A}$ is an open set.
In particular, a metric $g\in\mc{A}$ has no conjugate points (see \cite{wk}) and there is a unique geodesic $\gamma_c$ in each free-homotopy class $c\in \mc{C}$. We can thus define the \emph{marked length spectrum} of $g$ by \eqref{defLg}.

It will also be important for us to consider the mapping $g\mapsto L_g$ from the space of metrics 
to the set of sequences. In order to be in a good functional setting and since we shall work locally, 
we fix a smooth metric $g_0\in \mc{A}$ and consider the metrics $g$ in a neighborhood 
$\mc{U}_{g_0}$ of $g_0$ in $C^N(M;S^2_+T^*M)$ for some $N$ large enough and which will be chosen later. We can consider the map 
\begin{equation}\label{Lnormalise} 
\mc{L}: \mc{U}_{g_0}\to \ell^\infty(\mc{C}), \quad \mc{L}(g)(c):=\frac{L_g(c)}{L_{g_0}(c)}.
\end{equation} 
which we call the \emph{$g_0$-normalized marked length spectrum}. We notice from the definition of the length that $\mc{L}(g)\in [0,2]$ if $g\leq 2g_0$, justifying that $\mc{L}$ maps to $\ell^\infty(\mc{C})$.

\begin{prop}\label{LisC2} 
The functional \eqref{Lnormalise} is $C^2$ near $g_0$ if we choose the 
$C^3(M;S^2_+T^*M)$ topology. In particular, there is a neighborhood 
$\mc{U}_{g_0}\subset C^3(M;S^2_+T^*M)$ of $g_0$ and $C=C(g_0)>0$ such that for all $g\in \mc{U}_{g_0}$
\begin{equation}\label{taylor} 
\|\mc{L}(g)-\textbf{1}-D\mc{L}_{g_0}(g-g_0)\|_{\ell^\infty(\mc{C})}\leq C\|g-g_0\|^2_{C^3(M)}.
\end{equation}
\end{prop}
\begin{proof}
Let $\mc{M}:=S_{g_0}M$ be the unit tangent bundle for $g_0$ and $X_0$ the geodesic vector field.
We use the stability result in the work of De la Llave-Marco-Moryion \cite[Appendix A]{DMM} (see also the proof of \cite[Lemma 4.1]{DGRS})
which says that there is a neighborhood $\mc{V}_{X_0}$ in $C^2(\mc{M};T\mc{M})$ of $X_0$ and a $C^2$ map $X\in \mc{V}_{X_0}\mapsto \theta_X \in C^0(\mc{M})$ such that for each $X\in \mc{V}_{X_0}$ and each fixed periodic orbit $\gamma_{X_0}$ of $X_0$, there is a closed orbit $\gamma_X$ freely-homotopic to $\gamma_{X_0}$ 
and the period $\ell(\gamma_X)$ is $C^2$ as a map 
$X\in \mc{V}_{X_0}\mapsto \ell(\gamma_X) \in\rr^+$  given by 
\[ \ell(\gamma_X)=\int_{\gamma_{X_0}}\theta_X.\]
In particular, we see that $X\in \mc{V}_{X_0}\mapsto \ell(\gamma_X)/\ell(\gamma_{X_0})$ 
is $C^2$ and its derivatives of order $j=1,2$  are bounded:
\[ \|D^j\ell(\gamma_X)/\ell(\gamma_{X_0})\|_{C^2\to \rr}\leq  \sup_{X\in \mc{V}_{X_0}}\|D^j\theta_X\|_{C^2\to C^0}\leq C\]
for some $C$ depending on $\mc{V}_{X_0}$. 
Now  we fix $c\in \mc{C}$ and choose the geodesic $\gamma_c(g_0)$ for $g_0$ as being the element $\gamma_{X_0}$ above, and we take $\mc{U}_{g_0}$ a small neighborhood of $g_0$ in the $C^3$ topology. The map $X: g\in \mc{U}_{g_0}\mapsto X_{g}\in C^2(\mc{M};T\mc{M})$ 
is defined so that $X_g$ is the geodesic vector field of $g$, where we used the natural diffeomorphism between $\mc{M}=S_{g_0}M$ and $S_gM:=\{(x,v)\in TM; g_x(v,v)=1\}$ obtained by scaling the fibers to pull-back the field on $\mc{M}$. It is a $C^\infty$ map between the Banach space $C^3(M;S^2_+T^*M)$ and $C^2(\mc{M};T\mc{M})$. Thus the composition $g\mapsto \ell(\gamma_{X_g})$, which is simply 
the map $g\mapsto L_g(c)$, is $C^2$ on $\mc{U}_{g_0}$ and the second derivative
is uniformly bounded in $\mc{U}_{g_0}$. The inequality \eqref{taylor} follows directly.
\end{proof}

\subsection{The X-ray transform}\label{Xray}

The central object on which stands our proof is the \textit{X-ray transform} over symmetric $2$-tensors, which is nothing more than the linearization $D\mc{L}$ that appeared in Proposition \ref{LisC2}. It is a direct computation, which appeared already in \cite{GuKa}, that for 
$h\in C^3(M;S^2T^*M)$
\[ (D\mc{L}(g_0).h)(c)=\frac{1}{2 L_{g_0}(c)}\int_{0}^{L_{g_0}(c)}
h_{\gamma_c(t)}(\dot{\gamma}_c(t),\dot{\gamma}_c(t))dt\]
where $\gamma_c(t)$ is the geodesic for $g_0$ homotopic to $c$ and $\dot{\gamma}_c(t)$ its time derivative.
This leads us to define the so-called \emph{X-ray transform on $2$-tensors} for $g_0$ as the operator 
\begin{equation}\label{Xray}
I_2^{g_0}: C^3(M;S^2T^*M)\to \ell^\infty(\mc{C}), \qquad 
I_2^{g_0}h(c):= \frac{1}{L_{g_0}(c)}\int_{0}^{L_{g_0}(c)} h_{\gamma_c(t)}(\dot{\gamma}_c(t),\dot{\gamma}_c(t))dt
\end{equation}
Note that if $\varphi_t$ is the geodesic flow for $g_0$ (the flow of $X_{g_0}$), this can be rewritten as 
\[ I_2^{g_0}h(c)=\frac{1}{L_{g_0}(c)}\int_{0}^{L_{g_0}(c)} \pi_2^*h(\varphi_t(z))dt,\]
where $z\in \gamma_c$ is any point on the closed orbit and here, for $m\in \nn$, we denote by  $\pi_m^*$ is 
the natural continuous maps (for all  $k\in \nn\cup\{\infty\}$)
\[ \pi_m^* :C^k(M, S^m T^*M) \rightarrow C^k(SM), \qquad f \mapsto (\pi_m^* f) (x,v) = f(x)(\otimes^m v), \]
where we now use $SM$ as a notation for the unit tangent bundle for $g_0$.
More generally, if we define the \emph{X-ray transform} on $SM$ by 
\be 
I^{g_0} : C^0(SM) \rightarrow \ell^\infty(\mc{C}) ,\qquad
I^{g_0}h(c):= \frac{1}{L_{g_0}(c)}\int_{0}^{L_{g_0}(c)}h(\varphi_t(z))dt 
 \ee
with $z\in \gamma_c$, we will also define the \emph{X-ray transform on $m$-tensors} as the operator (for $m\in\nn$) defined on $C^0(M;S^mT^*M)$ by
\begin{equation}\label{Im}
I_m^{g_0}:= I^{g_0}\pi_m^*.
\end{equation}
When the background metric is fixed, we will remove the $g_0$ index and just write $I_m,I$ instead of $I_m^{g_0}, I^{g_0}$.
There is also a dual operator acting on distributions
\[ {\pi_m}_* : C^{-\infty}(SM) \rightarrow C^{-\infty}(M, S^m T^*M), \qquad 
\cjg {\pi_m}_*u,f\cjd:=\cjg u,\pi_m^*f\cjd \]
where $\cjg\cdot,\cdot\cjd$ denotes the distributional pairing. 
Let $H^s(SM)$ (resp. $H^s(M;S^mT^*M)$) denote the $L^2$-based Sobolev space of order $s\in\rr$ on $SM$ (resp. on $m$-tensors on $M$).
We note that for all $s\in\rr$, the following map are bounded
\begin{equation}\label{contpullback}
\pi_m^*: H^s(M;S^mT^*M)\to H^s(SM) ,\quad {\pi_m}_* :H^s(SM)\to H^s(M;S^mT^*M).
\end{equation}

Let us now explain the notion of \textit{solenoidal injectivity} of the X-ray transform. If $\nabla$ denotes the Levi-Civita connection of $g_0$ and $\sigma : \otimes^{m+1} T^*M \rightarrow S^{m+1} T^*M$ the symmetrisation operation, we define the symmetric derivative $D := \sigma \circ \nabla : C^\infty(M; S^m T^*M) \rightarrow C^\infty(M; S^{m+1} T^*M)$. The divergence operator is its formal adjoint given by $D^*f := -\text{Tr}(\nabla f)$, where $\text{Tr} :C^\infty(M; S^m T^*M) \rightarrow C^\infty(M; S^{m-2} T^*M)$ denotes the trace map defined by 
${\rm Tr}(q)(v_1, ...,v_{m-2}) = \sum_{i=1}^n q(e_i,e_i,v_1,...,v_{m-2}),$
if $(e_1,...e_n)$ is a local orthonormal basis of $TM$ for $g_0$.
If $f \in C^{k,\alpha}(M;S^m T^*M)$ with $(k,\alpha)\in \nn\x (0,1)$, there exists a unique decomposition of the tensor $f$ such that 
\be \label{eq:decomp} 
f = f^s + Dp,\quad  D^*f^s = 0, 
\ee
where $f^s \in C^{k,\alpha}(M;S^m T^*M)$ and $p \in 
C^{k+1,\alpha}(M; S^{m-1} T^*M)$ (see \cite[Theorem 3.3.2]{sh}). The tensor $f^s$ is called the divergence-free part (or solenoidal part) of $f$. 
It is direct to see that for each $f\in C^k(M;S^mT^*M)$, we have $\pi_{m+1}^*Df=X\pi_m^*f$ and that for $u\in C^k(SM)$ with $k\geq 1$ we have
$I(Xu)=0$ if $X=X_{g_0}$ is the geodesic vector field for $g_0$ on $SM$.  This implies that for $k\geq 1$
\[ \forall f\in C^k(M;S^mT^*M),\quad  I_{m+1}(Df)=0.\]
Thus in general it is impossible to recover the exact part $Dp$ of a tensor $f$ from $I_mf$.
We now recall some results about solenoidal injectivity of $I_m$, defined as the property
\begin{equation}\label{soleninj}
\ker I_m \cap C^\infty(M;S^mT^*M)\cap \ker D^*=0.
\end{equation}
\begin{prop}\label{injIm}
Let $(M,g_0)$ be a smooth Riemannian manifold and assume that the geodesic flow of $g_0$ is Anosov. 
Then $I_m$ is solenoidal injective in the sense \eqref{soleninj} when:
\begin{enumerate}
\item $m=0$ or $m=1$, see \cite[Theorem 1.1 and 1.3]{DaSh},
\item $m\in\nn$ and $\dim(M)=2$, see \cite[Theorem 1.4]{gu},\
\item $m\in\nn$ and $g_0$ has non-positive curvature, see \cite[Theorem 1.3]{cs}. 
\end{enumerate}
\end{prop}
The case (2) with $m=2$ was first proved in \cite[Theorem 1.1]{psu1}.

\section{The operator $\Pi$ and stability estimates}

In this section, we briefly review the results of the paper \cite{gu} and in particular the operator $\Pi$ defined there. As before we assume that $(M,g_0)$ has Anosov geodesic flow and let $X=X_{g_0}$ be its geodesic vector field.

\subsection{The operator $\Pi$}
Since $X$ preserves the Liouville measure $\mu$, the operator $-iX$ is an unbounded self-adjoint operator on $L^2(SM):=L^2(SM,d\mu)$. The $L^2$-spectrum is then contained in $\rr$ 
and the resolvents $R_\pm(\lambda):=(-X\pm\lambda)^{-1}$ are well-defined and bounded on $L^2(SM)$ for $\text{Re}(\lambda) > 0$, they are actually given by 
\[ R_+(\lambda)f(z)=\int_0^{+\infty} e^{-\lambda t}f(\varphi_t(z)) dt, \quad R_-(\lambda)f(z)=-\int_{-\infty}^{0} e^{\lambda t}f(\varphi_t(z)) dt \]
In \cite{fs}, Faure-Sj\"ostrand (see also \cite{BuLi,dz}) proved that for Anosov flows, there exists a constant $c > 0$ such that for any $s>0,r<0$, one can construct a Hilbert space $\mathcal{H}^{r,s}$ such that $H^s(SM) \subset \mathcal{H}^{r,s} \subset H^r(SM)$ and $-X-\lambda : \text{Dom}_{\mathcal{H}^{r,s}}(X) \rightarrow\mathcal{H}^{r,s}$ (with ${\rm Dom}_{\mathcal{H}^{r,s}}(X) :=\left\{u \in \mathcal{H}^{r,s}; X u \in \mathcal{H}^{r,s}\right\}$) is an unbounded Fredholm operator with index $0$ on $\text{Re}(\lambda) > -c\min(|r|,s)$; for $-X+\la$ the same holds with a Sobolev space $\mc{H}^{s,r}$ satisfying the same properties as above. Moreover, $-X \pm \lambda$ is invertible on these spaces for $\text{Re}(\lambda)$ large enough, the inverses coincide with $R_\pm(\lambda)$ when acting on $H^s(SM)$ and extend meromorphically to the half-plane $\text{Re}(\lambda) > -c\min(|r|,s)$, with poles of finite multiplicity. 

An Anosov geodesic flow is mixing \cite{An}, and $R_\pm(\la)$ has a simple pole at $\la=0$ with rank $1$ residue operator (\cite[Lemma 2.5]{gu}): one can then write the Laurent expansion\footnote{up to assuming that the Liouville measure $\mu$ is normalised to have mass $1$.}:
\be R_+(\lambda) = \dfrac{1 \otimes 1}{\lambda} + R_0 + \mathcal{O}(\lambda), ~~~ R_-(\lambda) = -\dfrac{1 \otimes 1}{\lambda} - R_0^* + \mathcal{O}(\lambda), \ee
where $R_0,R_0^* : H^s(SM) \rightarrow H^{r}(SM)$ are bounded. 
The operator $\Pi$ is then defined by:
\begin{equation}\label{defPi} 
\Pi := R_0 + R_0^* 
\end{equation}

The following Theorem was obtained by the first author in \cite{gu}:
\begin{prop}{\cite[Theorem 1.1]{gu}}
\label{th:gu1}
The operator $\Pi : H^s(SM) \rightarrow H^r(SM)$ is bounded, for any $s > 0, r < 0$, with infinite dimensional range, dense in the space of invariant distributions $C^{-\infty}_{\text{inv}}(SM):=\left\{w \in C^{-\infty}(SM); Xw =0\right\}$. It is a self-adjoint map $H^s(SM) \rightarrow H^{-s}(SM)$, for any $s > 0$, and satisfies:
\begin{enumerate}
\item $\forall f \in H^s(SM), X \Pi f = 0, $
\item $\forall f \in H^{s}(SM)$ such that $Xf\in H^s(SM)$, $\Pi X f = 0.$\footnote{In \cite{gu}, it is shown that $\Pi Xf=0$ if $f\in H^{s+1}(SM)$, but this implies the result by a density argument and the approximation result \cite[Lemma E.47]{dz2}.}
\end{enumerate}
If $f \in H^s(SM)$ with $\langle f,1\rangle_{L^2} = 0$, then $f \in \ke \Pi$ if and only if there exists a solution $u \in H^s(SM)$ to the cohomological equation $X u = f$,
and $u$ is unique modulo constants.
\end{prop}
We also add the following property 
\begin{equation}\label{Pion1}
\Pi1=0
\end{equation}
which follows directly from $R_\pm(\la)1=\pm 1/\la$.
The link between the X-ray transform $I$ and the operator $\Pi$ is rather unexplicit and given by the Livsic theorem \cite{anl}. For instance if $f \in C^\infty(SM)$ is in the kernel of the X-ray transform, i.e. $If = 0$, then we know by the smooth Livsic theorem that there exists $u \in C^\infty(SM)$ such that $f = Xu$ and thus $\Pi f = \Pi X u = 0$ by Theorem \ref{th:gu1}, (ii).

\begin{rema}
In the study of the X-ray transform on a manifold with boundary, it is natural to introduce the normal operator $I^*I$. It satisfies $XI^*Iu=0$, for any $u \in C^\infty$ and $I^*IXu = 0$ if $u$ vanishes on the boundary. For closed manifolds, the operator $\Pi$ is the replacement of the operator $I^*I$ used for manifolds with boundary (e.g. in \cite{PeUh,gu2,SUV,tl}).
\end{rema}

\subsection{The operators $\Pi_m$}
For $m\in\nn$, we introduce the operator $\Pi_m := {\pi_m}_* \Pi \pi_m^*$ mapping 
 $C^\infty(M;S^mT^*M)$ to $C^{-\infty}(M;S^mT^*M)$. In \cite{gu}, the first author
 studied the microlocal properties of $\Pi_m$ by using in particular the works \cite{fs,dz}.

\begin{prop}{\cite[Theorem 3.5, Lemma 3.6]{gu}}\footnote{In Lemma 3.6 in \cite{gu}, there is a typo as for injectivity of $\Pi \pi_m^*$, one needs to acts on distributions vanishing on constants.}
\label{th:gu2}
The operator $\Pi_m$ is a pseudodifferential operator of order $-1$ which is elliptic on solenoidal tensors in the sense that there exists pseudodifferential operators $Q, S, R$ of respective order $1, -2, -\infty$ such that:
\[ Q\Pi_m= {\rm Id} + DSD^* + R \]
Moreover for any $s > 0$, $\Pi \pi_m^* : H^{-s}(M;S^m T^* M)\cap \ker D^* \rightarrow H^{-s}(SM)$ is bounded. When restricted to $\{f\in H^{-s}(M;S^m T^* M); f\in \ker D^*, \cjg \pi_m^*f,1\cjd_{L^2(SM)}=0\}$ it is injective if $I_m$ is solenoidal injective in the sense of \eqref{soleninj}.
\end{prop}
This results implies the following stability estimates.
\begin{lemm}
\label{lem:fun2}
Assume that $I_m$ is solenoidal injective in the sense \eqref{soleninj}.
For all $s > 0$, there exists a constant $C > 0$ depending on $g_0, s$ such that for all 
$ f \in H^{-s}(M;S^m T^* M)\cap \ker D^*$:
\begin{equation}  \label{eq:fun2} 
 \|f\|_{H^{-s-1}(M)} \leq C (\|\Pi \pi_m^* f\|_{H^{-s}(M)}+|\cjg \pi_m^*f,1\cjd_{L^2(SM)}|) 
\end{equation}
\end{lemm}
\begin{proof}
This is actually a consequence of Proposition \ref{th:gu2}. We know that there exist pseudodifferential operators $Q, S, R$ of respective order $1, -2, -\infty$ on $M$ such that:
\[Q\Pi_m = \text{Id} + DSD^* + R.\]
For each $f \in H^{-s}(M;S^m T^* M)$ with $D^*f=0$, we have $\Pi \pi_2^* f \in H^{-s}(SM)$ by Theorem \ref{th:gu2}. Then then exists $C>0$ (which may change from line to line) such that for all such $f$
\[Â \begin{split} 
\|f\|_{H^{-s-1}} & \leq C(\|Q\Pi_mf\|_{H^{-s-1}} + \|Rf\|_{H^{-s-1}})  \leq 
C(\|{\pi_m}_*\Pi\pi_m^* f\|_{H^{-s}} + \|Rf\|_{H^{-s-1}}) \\ & \leq C(||\Pi \pi_m^*f\|_{H^{-s}} + \|Rf\|_{H^{-s-1}}). 
\end{split} \]
where we used \eqref{contpullback} and the boundedness of pseudodifferential operators on Sobolev spaces. The proof now boils down to a standard argument of functional analysis. Assume (\ref{eq:fun2}) does not hold. Then, one can find a sequence of tensors $f_n \in H^{-s}(M; S^m T^* M)\cap \ker D^*$, such that $\|f_n\|_{H^{-s-1}}=1$ and thus:
\[1 = \|f_n\|_{H^{-s-1}} \geq n(\|\Pi\pi^*_mf_n\|_{H^{-s}}+|\cjg \pi_m^*f_n,1\cjd_{L^2}|),\] 
that is $\Pi \pi^*_m f_n \rightarrow_{n \rightarrow \infty} 0$ in $H^{-s}$ (and thus in particular in $H^{-s-1}$) and $\cjg \pi_m^*f_n,1\cjd_{L^2}\to 0$. Since $R$ is compact and $(f_n)_{n \in \N}$ is bounded in $H^{-s-1}$, we can assume (up to extraction) that $Rf_n \rightarrow_{n \rightarrow \infty} v \in H^{-s-1}$. By the previous inequality, we deduce that $(f_n)_{n \in \N}$ is a Cauchy sequence in $H^{-s-1}$, which thus converges to an element $f \in H^{-s-1}(M;S^m T^* M)\cap \ker D^*$ such that $\|f\|_{H^{-s-1}}=1$ and $\cjg \pi_m^*f,1\cjd_{L^2}=0$. The operator $\Pi \pi_m^* : H^{-s-1} \rightarrow H^{-s-1}$ is bounded by Proposition \ref{th:gu2} 
so $\Pi \pi_m^* f_n \rightarrow_{n \rightarrow \infty} 0 = \Pi \pi_m^* f$ and it is also injective on 
$\ker D^*\cap \{f; \cjg \pi_m^*f,1\cjd_{L^2}=0\}$ so $f \equiv 0$. This is a contradiction.
\end{proof}
\begin{rema}
With a bit more work, we can actually get a better estimate with a ${-(s+1/2)}$ Sobolev exponent on the left-hand side of \eqref{eq:fun2}. 
\end{rema}

\section{Proofs of the main results}

As before, we fix a smooth Riemannian manifold $(M,g_0)$ with Anosov flow and will shall consider 
metrics $g$ with regularity $C^{N,\alpha}$ for some $N\geq 3, \alpha>0$ to be determined later and
 such that $\|g-g_0\|_{\mathcal{C}^{N,\alpha}} < \eps$, for some $\eps>0$ small enough so that 
 $g$ also has Anosov flow.

\subsection{Reduction of the problem}

The metric $g_0$ is divergence-free with respect to itself: $D^*g_0 = -{\rm Tr}(\nabla g_0) = 0$, 
where the Levi-Civita connection $\nabla$ and trace ${\rm Tr}$ are defined with respect to $g_0$. By a standard argument, there is a slice transverse to the diffeomorphism action
$(\phi,g)\mapsto \phi^*g$ at the metric $g_0$; here $\phi$ varies in the group  
of $C^{N,\alpha}$-diffeomorphisms on $M$ homotopic to the identity. 
We shall write ${\rm Diff}_0^{N,\alpha}(M)$ for the group of $C^{N,\alpha}(M)$ diffeomorphisms homotopic to ${\rm Id}$, with $N\geq 2,\alpha\in(0,1)$. 
Since $\mc{L}(\phi^*g_0)=\mc{L}(g_0)=\textbf{1}$ for all $\phi\in {\rm Diff}^{N,\alpha}_0(M)$, it suffices to work on that transverse slice to study the marked length spectrum. This is the content of the following: 
\begin{lemm}{\cite[Theorem 2.1]{cds}}
\label{lemm:reduction}
Let $N\geq 2$ be an integer, $\alpha\in (0,1)$. For any $\delta > 0$ small enough, there exists $\eps > 0$ such that for any $g$ satisfying $\|g-g_0\|_{C^{N,\alpha}} < \eps$, there exists $\phi\in {\rm Diff}^{N,\alpha}_0(M)$ that is $C^{N,\alpha}$ close to ${\rm Id}$ such that $g' := \phi^*g$ 
is divergence-free with respect to the metric $g_0$ and $\|g'-g_0\|_{C^{N,\alpha}}< \delta$.
\end{lemm} 
We introduce $f:=\phi^*g-g_0\in C^{N,\alpha}(M;S^2T^*M)$, which is, by construction, divergence-free and satisfies $\|f\|_{C^{N,\alpha}} < \delta$. Our goal will be to prove that $f \equiv 0$, if $\delta$ is chosen small enough and $L_g=L_{g_0}$.

\subsection{Geometric estimates}
We let $g,g_0$ be two $C^3$-metrics with Anosov geodesic flow. 
\begin{lemm}
\label{lem:geo1}
Assume that $L_{g}(c)\geq L_{g_0}(c)$ for each $c\in \mc{C}$. If 
$\gamma_c$ denotes the unique geodesic freely homotopic to $c$ for $g_0$, then 
\[ I_2^{g_0}f(c)=\int_{\gamma_c} \pi_2^*f \geq 0.\]
\end{lemm}
\begin{proof}
We denote by $\gamma'_c$ the $g$-geodesic in the free-homotopy class $c$. One has:
\[ \int_{\gamma_c} \pi_2^*f=\int_{\gamma_c} \pi_2^*g-\int_{\gamma_c}\pi_2^*g_0=E_g(\gamma_c)-L_{g_0}(c),\]
where $E_g(\gamma_c)=\int_{0}^{\ell_{g_0}(\gamma_c)}g_{\gamma_c(t)}(\dot{\gamma}_c(t),\dot{\gamma}_c(t))dt$ is the energy functional for $g$. By using Cauchy-Schwartz, $E_g(\gamma_c)\geq \ell_g(\gamma_c)^2/\ell_{g_0}(\gamma_c)$ and
since $\gamma_c$ is freely-homotopic to $c$, we get $\ell_g(\gamma_c)\geq \ell_g(\gamma'_c)$. Since $\ell_g(\gamma'_c)\geq \ell_{g_0}(\gamma_c)$ by assumption, we obtain the desired inequality.
\end{proof}

Next, we can use the following result 
\begin{lemm}
\label{lem:geo2}
There exists $\eps>0$ small such that if $\|g-g_0\|_{C^0}\leq \eps$ and 
${\rm Vol}_{g}(M)\leq {\rm Vol}_{g_0}(M)$, then with $f=g-g_0$
\[ \int_{SM}\pi_2^*f \, d\mu\leq \frac{2}{3}\|f\|_{L^2}^2.\]
\end{lemm}
\begin{proof}
Let $g_\tau := g_0 + \tau f$ with $f\in C^3(M;S^2T^*M)$.
A direct computation gives that $\int_{M}{\rm Tr}_{g_0}(f){\rm dvol}_{g_0}=\int_{SM}\pi_2^*f \, d\mu$. Then the argument of \cite[Proposition 4.1]{cds} by Taylor expanding ${\rm Vol}_{g_\tau}(M)$
in $\tau$ directly provides the result.
\end{proof}
Finally, we conclude this section with the following: 
\begin{lemm}
\label{lem:geo3} Assume that $I^{g_0}_2f(c)\geq 0$ for all $c\in \mc{C}$. Then,
there exists a constant $C=C(g_0) > 0$, such that:
\begin{equation}\label{boundonintf}
Â 0 \leq \int_{SM} \pi_2^*f \, d\mu \leq C\left(\|\mc{L}(g) - \mathbf{1}\|_{\ell^{\infty}(\mc{C})} +  \|f\|^2_{C^3(M)}\right) 
\end{equation}
where $d\mu$ is the Liouville measure of $g_0$ and $g=g_0+f$ as above.
\end{lemm}
\begin{proof}
For the Anosov geodesic flow of $g_0$, the Liouville measure is the unique equilibrium state associated to the potential given by $J^u(z) := -\partial_t \left( \det d\varphi_t(z)|_{E_u(z)}\right)|_{t=0}$ (the unstable Jacobian). By Parry's formula (see \cite[Paragraph 3]{wp}), we have: 
\be \label{eq:parry} 
\forall F \in C^0(SM), \quad \lim_{T\to \infty}\dfrac{1}{N(T)} \sum_{c \in \mathcal{C}, L_{g_0}(c)\leq T} 
\dfrac{e^{\int_{\gamma_c} J^u }}{L_{g_0}(c)} \int_{\gamma_c} F 
=\frac{1}{{\rm Vol}(SM)}\int_{SM}F\, d\mu,
\ee
where, as before, $\gamma_c$ is the $g_0$-geodesic in $c$ and $N(T)$ is the constant of normalisation corresponding to the sum when $F=1$. The first inequality in \eqref{boundonintf} then follows from that formula and the assumption $I^{g_0}_2f\geq 0$. For the second inequality in \eqref{boundonintf} 
we use Proposition \ref{LisC2} with the fact that $D\mc{L}_{g_0}f=\frac{1}{2} I_2^{g_0}f$ to deduce that there exists $C(g_0)>0$ such that
\begin{equation}\label{tayloronL} 
\|I_2^{g_0}f\|_{\ell^{\infty}(\mc{C})} \leq 2 \|\mc{L}(g) - \mathbf{1}\|_{\ell^{\infty}(\mc{C})} 
+ C(g_0)\|f\|_{C^3}^2.\end{equation}
Thus, we get for any $T > 0$
\begin{equation}\label{rapportI2intf} 
 \dfrac{1}{N(T)} \sum_{c \in \mathcal{C}, L_{g_0}(c)\leq T} 
e^{\int_{\gamma_c} J^u }I_2^{g_0}f(c)\leq \|I_2^{g_0}f\|_{\ell^\infty(\mc{C})} \leq 2 \|\mc{L}(g) - \mathbf{1}\|_{\ell^{\infty}(\mc{C})} + C(g_0)\|f\|_{C^3}^2  \end{equation}
and the left-hand side converges to $\frac{1}{{\rm Vol}(SM)} \int_{SM} \pi_2^*f \, d\mu$ by Parry's formula (\ref{eq:parry}), which is the sought result by letting $T\to \infty$.
\end{proof}
We note that in the previous proof, the approximation of $\int_{SM}\pi_2^*f$ by $I_2^{g_0}f(c)$ could also be done using the Birkhoff ergodic theorem and the Anosov closing lemma to approximate $\int_{SM}\pi_2^*f$ by some $I^{g_0}_2f(c)$ for some $c\in\mc{C}$ so that $L_{g_0}(c)$ is large.

The following lemma is another key ingredient in the proof of our main results. It is a positive version of Livsic theorem which was proved independently by Pollicott-Sharp \cite{ps} and Lopes-Thieullen \cite{lt} (though the stronger version we use is actually that of \cite{lt}). Here $\mathcal{M}_1$ denotes the Borel probability measures on $SM$ which are invariant by the geodesic flow of $g_0$. Note that, by Sigmund \cite{Si}, the Dirac measures $u\mapsto \frac{1}{L_{g_0}(c)}\int_{\gamma_c}u$ on closed orbits are dense in $\mc{M}_1$.
\begin{prop}{\cite[Theorem 1]{lt}, \cite[Theorem 1]{ps}}
\label{prop:geo4}
Let $\alpha \in (0,1]$ and let $X_0$ be the geodesic vector field of $g_0$. There exists a constant 
$C=C(g_0) > 0$ and $\beta\in (0,1)$ such that for any $u \in C^\alpha(SM)$ satisfying
\[\forall c\in  \mathcal{C},\quad \int_{\gamma_c} u  \geq 0,\]
there exists  $h\in C^{\alpha\beta}(SM)$ and $F\in C^{\alpha\beta}(SM)$ such that $F \geq 0$ and  $u + Xh = F$. Moreover 
$\|F\|_{C^{\alpha\beta}} \leq C\|u\|_{C^\alpha}$.
\end{prop}

\subsection{Proof of Theorem \ref{th1} and Theorem \ref{th1bis}}
\label{ssect:proofth1}
We fix $g_0$ with Anosov geodesic flow on $M$ and assume that either $M$ is a surface or that $g_0$ has non-positive curvature in order to have
that $I_2^{g_0}$ is solenoidal injective by Proposition \ref{injIm}. Fix $N\geq 3$ to be chosen later and $\alpha>0$ small.
As explained in Lemma \ref{lemm:reduction}, we take $\delta>0$ small and $\eps>0$ 
small so that $\|g-g_0\|_{C^{N,\alpha}}<\eps$ 
implies that there is $\phi\in {\rm Diff}_0^{N,\alpha}(M)$ with 
$\|\phi^*g-g_0\|_{C^{N,\alpha}}<\delta$ and $D^*(\phi^*g-g_0)=0$.

We write $f:=\phi^*g-g_0$ and remark that the assumption $L_{g}\geq L_{g_0}$ implies 
$L_{\phi^*g}\geq L_{g_0}$ thus $I_2^{g_0}f(c)\geq 0$ for all $c\in \mc{C}$ by Lemma \ref{lem:geo1}.
By Proposition \ref{prop:geo4}, we know that there exists $h\in C^{\beta}(SM)$ and $F\in C^{\beta}(SM)$ for some $0<\beta<\alpha$ (depending on $g_0$ and linearly on $\alpha$) such that $\pi_2^*f + Xh = F \geq 0$, with
\be \label{eq:control} 
\|\pi_2^*f + Xh\|_{C^{\beta}}  \leq C\|\pi_2^*f\|_{C^\alpha} 
\leq C\|f\|_{C^\alpha},
\ee
where $C=C(g_0)$. 
Take $0<s\ll \beta$ very small (it will be fixed later) and let $\beta'<\beta$ be very close to $\beta$. Thus we obtain (for some constant $C=C(g_0,s,\beta)$ that may change from line to line)
\begin{equation}\label{seqineq}
\begin{array}{lll} 
\|f\|_{H^{-1-s}} & \leq C(\|\Pi \pi_2^*f\|_{H^{-s}}+|\cjg \pi_2^*f,1\cjd_{L^2}|), & \text{ by Lemma \ref{lem:fun2}} \\ 
& \leq C(\| \Pi (\pi_2^*f + Xh)\|_{H^{-s}}+|\cjg \pi_2^*f + Xh,1\cjd_{L^2}|), &\text{ since } \Pi X h = 0 \\
& \leq C\|\pi_2^*f + Xh\|_{H^s}, &\text{ by Theorem \ref{th:gu1}} \\
& \leq C\|\pi_2^*f + Xh \|_{L^2}^{1-\nu} \|\pi_2^*f + Xh\|_{H^{\beta'}}^{\nu} , & \text{ by interpolation with } \nu = \frac{s}{\beta'} .\end{array}\end{equation}
Note that by (\ref{eq:control}) we have a control:
\begin{equation}\label{boundLoTh} 
\|\pi_2^*f + Xh\|_{H^{\beta'}} \leq C \|\pi_2^*f + Xh\|_{C^{\beta}} \leq C\|f\|_{C^\alpha} .
\end{equation}
And we can once more interpolate between Lebesgue spaces so that:
\begin{equation}\label{ineqpi2*f}
\|\pi_2^*f + Xh\|_{L^2} \leq C\|\pi_2^*f + Xh\|_{L^1}^{1/2} \|\pi_2^*f + X h\|_{L^\infty}^{1/2} 
\leq C\|\pi_2^*f + Xh\|_{L^1}^{1/2} \|f\|_{C^\alpha}^{1/2}.
\end{equation}
Next, using that $\pi_2^*f + Xh\geq 0$, we have 
\beÂ \begin{split} \label{eq:positivity} 
\|\pi_2^*f + Xh\|_{L^1} = \int_{SM} (\pi_2^*f + Xh)d\mu =
 \int_{SM} \pi_2^*f \, d\mu  
 \end{split}\ee
We will now consider two cases: in case (1) we assume that $L_{g}=L_{g_0}$, while in case (2) we assume that ${\rm Vol}_{g}(M)\leq {\rm Vol}_{g_0}(M)$.
Combining Lemma \ref{lem:geo1} and Lemma \ref{lem:geo3}, we deduce that
in case (1),  we have
\[\|\pi_2^*f + Xh\|_{L^1}\leq C\|f\|^2_{C^3},
\]
while in case (2), we get by Lemma \ref{lem:geo2} that if $\eps>0$ is small enough,
\[\|\pi_2^*f + Xh\|_{L^1}\leq C\|f\|^2_{L^2}\]
These facts combined with \eqref{ineqpi2*f}  yield 
\[\|\pi_2^*f + Xh\|_{L^2} \leq \left\{\begin{array}{ll}
C\|f\|_{C^3}.\|f\|_{C^\alpha}^{1/2}, & \textrm{case (1)}\\ 
C\|f\|_{L^2}.\|f\|_{C^\alpha}^{1/2}, & \textrm{case (2)}\end{array}\right.
\]
Thus we have shown 
\begin{equation}\label{boundnonlin} 
\|f\|_{H^{-1-s}}\leq \left\{\begin{array}{ll}
C\|f\|_{C^3}^{1-\nu}\|f\|_{C^\alpha}^{\frac{1+\nu}{2}}, & \textrm{case (1)}\\ 
C\|f\|^{1-\nu}_{L^2}.\|f\|_{C^\alpha}^{\frac{1+\nu}{2}}, & \textrm{case (2)}
\end{array}\right.
\end{equation}
where $C=C(g_0,s,\beta)$.
We choose $\alpha$ very small and $0<s\ll \beta<\alpha$, $j\in\{\alpha,3\}$
and $N_0>n/2+j+s$: by interpolation and Sobolev embedding we have 
\begin{equation}\label{interpol} 
\|f\|_{C^j}Â \leq \|f\|_{H^{n/2+j+s}} \leq C \|f\|^{1-\theta_j}_{H^{-1-s}}\|f\|^{\theta_j}_{H^{N_0}} 
\end{equation}
with $\theta_j=\frac{n/2+j+1+2s}{N_0+s+1}$. If $N_0>\frac{3}{2}n+8$, we see that 
$\gamma:=\demi(1-\theta_\alpha)(1+\nu)+(1-\theta_3)(1-\nu)>1$ if $s>0$ and $\alpha$ are chosen small enough, thus in case (1) we get with $\gamma':=(1+\nu)\theta_\alpha/2+(1-\nu)\theta_3$
\[ \|f\|_{H^{-1-s}}\leq C\|f\|^{\gamma}_{H^{-1-s}}\|f\|_{H^{N_0}}^{\gamma'}\]
Thus if $f\not=0$ we obtain, if $\|f\|_{H^{N_0}}\leq \delta$
\[ 1\leq C\|f\|_{H^{-1-s}}^{\gamma-1}\|f\|_{H^{N_0}}^{\gamma'}\leq C\|f\|_{H^{N_0}}^{\gamma-1+\gamma'}\leq C\delta^{\gamma-1+\gamma'}.
\]
Since $\gamma-1+\gamma'>0$, we see that by taking $\delta>0$ small enough we obtain a contradiction, thus $f=0$. This proves Theorem \ref{th1} by choosing $N\geq N_0$.
In case (2) (corresponding to Theorem \ref{th1bis}), this is the same argument except that we get a slightly better result due to the $L^2$ norm in \eqref{boundnonlin}: $N_0$ can be chosen to be any number $N_0>n/2+2$. To conclude, we have shown that if  
$\|g-g_0\|_{C^{N,\alpha}}<\eps$ for $N\in \nn$ with $N+\alpha>n/2+2$, then $L_g\geq L_{g_0}$ implies that either ${\rm Vol}_{g}(M)\leq {\rm Vol}_{g_0}(M)$ and $\phi^*g=g_0$ for some $C^{N,\alpha}$ diffeomorphism, or  ${\rm Vol}_{g}(M)\geq {\rm Vol}_{g_0}(M)$.
Note that in both cases, if $g$ is smooth then $\phi$ is smooth.

\subsection{Stability estimates for X-ray transforms}

We next give some new stability estimates for the X-ray transform. To the best of our knowledge, these are the first estimates in the closed setting. 
In the following, we will consider the X-ray transform $I_m$ over divergence-free symmetric $m$-tensors.
\begin{theo}
\label{th:stab1}
Assume that $(M,g_0)$ satisfies the assumptions of Theorem \ref{th1}. Then for all $\alpha> 0$, there is $\beta\in(0,\alpha)$ depending  linearly on $\alpha$ such that for all $s\in (0,\beta)$ 
and for all $\nu\in(s/\beta,1)$,  there exists a constant $C > 0$ such that for all 
$f \in C^\alpha(M; S^m T^* M)\cap \ker D^*$:
\[ \|f\|_{H^{-1-s}} \leq C\|I^{g_0}_mf\|_{\ell^\infty}^{(1-\nu)/2} (\|f\|_{C^\alpha}+\|I_m^{g_0}f\|_{\ell^\infty})^{(1+\nu)/2}\]
\end{theo}
\begin{proof}
The proof is essentially the same as Theorem \ref{th1}. By using Lemma \ref{lem:fun2} with $\pi_m^*f$ replaced by $\pi_m^*f+\|I^{g_0}_mf\|_{\ell^\infty}$ and Proposition \ref{prop:geo4}, we 
have, as in \eqref{seqineq}, that for all $0<\alpha<1$ small, there is $0<\beta<\alpha$ depending on $g_0$ and linearly on $\alpha$ such that for all $0<s<\beta'<\beta$, and for all 
$f\in C^{\alpha}(M;S^mT^*M)\cap \ker D^*$:
\[\begin{split}
\|f\|_{H^{-1-s}} \leq & C(\|\Pi \pi_m^*f\|_{H^{-s}}+|\cjg \pi_m^*f,1\cjd_{L^2}|)
\leq C(\|\Pi(\pi_m^*f+Xh)\|_{H^{-s}}+|\cjg \pi_m^*f+Xh,1\cjd_{L^2}|)\\
\leq  & C\|\pi_m^*f + Xh \|_{L^2}^{1-\nu} \|\pi_m^*f + Xh\|_{H^{\beta'}}^{\nu},
\end{split}\]
for some $C$ depending only on $(g_0,s,\beta,\beta',\alpha)$,  $\nu:=s/\beta'$ and
where $\pi_m^*f+Xh=-\|I_m^{g_0}f\|_{\ell^\infty}+F$ with $h,F\in C^{\beta}$ such that $\|F\|_{C^{\beta}}\leq C(\|f\|_{C^\alpha}+\|I^{g_0}_mf\|_{\ell^\infty})$.
Using \eqref{boundLoTh}, \eqref{ineqpi2*f}, \eqref{eq:positivity}, \eqref{rapportI2intf} with $I_2^{g_0}f$ replaced by $I_m^{g_0}f$, we get the result.
\end{proof}

\begin{rema}\label{kerI2An} Note that $\nu$ and $s$ can be chosen arbitrarily small in the estimate.
In the general case of an Anosov manifold (without any assumption on the curvature), the $s$-injectivity of the X-ray transform is still unknown. However, it was proved in \cite[Theorem 1.5]{DaSh} and \cite[Lemma 3.6]{gu}  that its kernel is finite-dimensional and contains only smooth tensors. The same arguments as above then show that Theorem \ref{th:stab1} still holds for all $f$ as above with the extra condition $f\perp \ker I_m$ with respect to the $L^2$ scalar product, and similarly for Theorem \ref{th1}, if $g$ is not in a finite dimensional manifold.
\end{rema}

\subsection{Stability estimates for the marked length spectrum. Proof of Theorem \ref{th3}}
We will apply the same reasoning as before to get a stability estimate for the non-linear problem (the marked length spectrum).
We proceed as before and reduce to considering $f=\phi^*g-g_0$ where 
$\phi\in {\rm Diff}_0^{N,\alpha}(M)$ and $\|f\|_{C^{N,\alpha}}<\delta$.
By Theorem \ref{th:stab1}, and using \eqref{tayloronL} we have for $0<\alpha$ small, 
$0<s\ll \alpha$ and $\beta,\nu$ as in Theorem \ref{th:stab1} (in particular $\nu,\alpha,s$ can be made arbitrarily small):
\[ \begin{split} 
\|f\|_{H^{-s-1}} & \leq C \|I^{g_0}_2f\|_{\ell^\infty}^{(1-\nu)/2} (\|f\|_{C^\alpha} + \|I^{g_0}_2f\|_{\ell^\infty})^{(1+\nu)/2}  \\
& \leq C(\|\mc{L}(g)-\textbf{1}\|_{\ell^\infty}+\|f\|^2_{C^3})^{(1-\nu)/2}\|f\|_{C^\alpha}^{(1+\nu)/2}+
C(\|\mc{L}(g)-\textbf{1}\|_{\ell^\infty}+\|f\|^2_{C^3})\\
& \leq C \left(\|\mc{L}(g)- \mathbf{1}\|_{\ell^\infty}^{(1-\nu)/2} \|f\|^{(1+\nu)/2}_{C^\alpha} + 
\|\mc{L}(g)- \mathbf{1}\|_{\ell^\infty} + \|f\|^{1-\nu}_{C^3}\|f\|^{(1+\nu)/2}_{C^\alpha} \right)
\end{split} \]
We use the interpolation estimate \eqref{interpol}
and for $N_0>n/2+9$ we get
\be \label{eq:interp1} 
\|f\|_{H^{-s-1}} \leq C \left(\|\mc{L}(g)- \mathbf{1}\|_{\ell^\infty}^{(1-\nu)/2} \|f\|^{(1+\nu)/2}_{C^\alpha} + \|\mc{L}(g)- \mathbf{1}\|_{\ell^\infty}+ \|f\|^{\gamma}_{H^{-s-1}} \|f\|^{\gamma'}_{C^{N_0}} \right), 
\ee
where $\gamma=\demi(1-\theta_\alpha)(1+\nu)+(1-\theta_3)(1-\nu)>1$, $\gamma'>0$, $\theta_3=\frac{n/2+4+2s}{N_0+s+1}$, if $s>0$ is chosen small enough.
Assume $\delta$ is chosen small enough so that $C \delta^{\alpha /2} \leq 1/2$. Then:
\[ \begin{split} 
\|f\|^{\gamma}_{H^{-s-1}} \|f\|^{\gamma'}_{C^{N_0}} 
\leq C\|f\|_{H^{-s-1}}  \|f\|_{C^{N_0}}^{(\gamma-1)+\gamma'} 
\leq C\|f\|_{H^{-s-1}}\delta^{(\gamma-1)+\gamma'}\leq \demi \|f\|_{H^{-s-1}},
\end{split} \]
if $\delta>0$ is chosen small enough depending on $C=C(g_0,s,\alpha,\beta,\nu)$ and $N+\alpha>N_0$. The sought result then follows from the previous inequality combined with (\ref{eq:interp1}).

\subsection{Compactness theorems and proof of Theorem \ref{finiteness}}

We let $M$ be a compact smooth manifold equipped with an Anosov geodesic flow. By the proof of  
\cite[Theorem 4.8]{Kn}, the universal cover $\til{M}$ and the fundamental group $\pi_1(M)$ are hyperbolic in the sense of Gromov \cite{Gr}. We shall denote by $\mathcal{R}_g$ the curvature tensor associated to the metric $g$ and by ${\rm inj}(g)$ the injectivity radius of $g$. 
We proceed by contradiction: let $(g_n)_{n\geq 0}$ be a sequence of smooth metrics on $M$ in the class $\mc{A}(\nu_0,\nu_1,C_0,\theta_0)$ (defined in the Introduction) such that $L_{g_n}=L_{g_0}$ and such that for each $k\in\nn$
there is $B_k>0$ such that $|\nabla_{g_n}^k\mc{R}_{g_n}|_{g_n}\leq B_k$ for all $n$, and we assume that 
for each $n\not=n'$, $g_n$ is not isometric to $g_{n'}$.
Since the metrics have Anosov flow, they have no conjugate points and thus  
\[{\rm inj}(g_n)= \demi \min_{c\in \mc{C}} L_{g_n}(c)=\demi \min_{c\in \mc{C}} L_{g_0}(c).\]
By Hamilton compactness result \cite[Theorem 2.3]{Ham}, there is a family of smooth diffeomorphisms $\phi_n$ on $M$ 
such that $g_n':=\phi_n^*g_n$ converges to $g\in \mc{A}(\nu_0,\nu_1,C_0,\theta_0)$ in the $C^\infty$ topology (note that $\mc{A}(\nu_0,\nu_1,C_0,\theta_0)$ is invariant by pull-back through smooth diffeomorphisms). 
Denote by ${\phi_n}_*\in {\rm Out}(\pi_1(M))$ the action of $\phi_n$ on the set of conjugacy classes $\mc{C}$.
The universal cover $\til{M}$ of $M$ is a ball since $M$ has no conjugate points, and $\pi_1(M)$ is a hyperbolic group thus we can apply the result of Gromov \cite[Theorem 5.4.1]{Gr} saying that the outer automorphism group 
${\rm Out}(\pi_1(M))$ is finite if $\dim M\geq 3$. This implies in particular that there is 
a subsequence $(\phi_{n_j})_{j\in\nn}$ such that ${\phi_{n_j}}_*(c)={\phi_{n_0}}_*(c)$ 
for all $c\in \mc{C}$ and all $j\in\nn$ where as before $\mc{C}$ is the set of conjugacy classes of $\pi_1(M)$. But $\phi_{n_0}^*g_{n_j}$ have the same marked length spectrum as 
$\phi^*_{n_0}g_0$ for all $j$, thus $L_{g_{n_j}'}=L_{\phi^*_{n_0}g_0}$ for all $j$. 
Since $g'_{n_j}\to g$ in $C^\infty$, we have $L_{g}=L_{g_{n_j}'}$ for all $j$ and by Theorem \ref{th1}, we deduce that there is $j_0$ such that for all $j\geq j_0$, $g'_{n_j}$ is isometric to $g$. This gives a contradiction.
 
Now, if $\dim M=2$, ${\rm Out}(\pi_1(M))$ is a discrete infinite group. 
 We first show that for each $c\in \mc{C}$, the set of classes $(\phi_n^{-1})_*(c)\in \mc{C}$ 
 is finite as $n$ ranges over $\nn$. Assume the contrary, then consider $\gamma_n$ the geodesic for $g_n$ in 
the class $c$, one has $L_{g_n}(c)=\ell_{g_n}(\gamma_n)=\ell_{g_0}(\gamma_0)$, by assumption.
Now $\phi_n^{-1}(\gamma_n)$ is a $g_n'$ geodesic in the class $(\phi_n^{-1})_*(c)$ with length 
$\ell_{g_n'}(\phi_n^{-1}(\gamma_n))=\ell_{g_n}(\gamma_n)=\ell_{g_0}(\gamma_0)$. We know that there are finitely many $g$-geodesics with length less than $\ell_{g_0}(\gamma_0)$, but we also have  
\[ L_g((\phi_n^{-1})_*(c))\leq \ell_g(\phi_n^{-1}(\gamma_n))\leq \ell_{g_n'}(\phi_n^{-1}(\gamma_n))(1+\eps)\leq \ell_{g_0}(\gamma_0)(1+\eps),\] 
if $\|g_n'-g\|_{C^3}\leq \eps$. Thus we obtain a contradiction for $n$ large. The extended mapping class group\footnote{extended in the sense that it includes orientation reversing elements.} ${\rm Mod}(M)$ is isomorphic to
${\rm Out}(\pi_1(M))$ (see \cite[Theorem 8.1]{FaMa}). By \cite[Proposition 2.8]{FaMa}\footnote{see also the proof of Theorem 3.10 in \cite{FaMa}}, if $M$ has genus at least $3$, there is a finite set $\mc{C}_0\subset \mc{C}$ such that if $\phi_*\in {\rm Mod}(M)$  is the identity on $\mc{C}_0$ then  
$\phi$ is homotopic to ${\rm Id}$, while if $M$ has genus $2$, the same condition implies that $\phi$ is either homotopic to ${\rm Id}$ or to an hyperelliptic involution $h$. In both cases, we can extract a subsquence $\phi_{n_j}$ such that 
${\phi_{n_j}}_*={\phi_{n_0}}_*$ for all $j\geq 0$ and we conclude like in the higher dimensional case.

\nocite{*}
\bibliographystyle{cdraifplain}
\bibliography{xampl}

\begin{thebibliography}{9}

\bibitem[An]{An} D. V. Anosov, \emph{Geodesic flows on closed Riemannian manifolds with negative curvature}, 
Trudy Mat. Inst. Steklov. \textbf{90} (1967) 209 pp.


\bibitem[BCG]{BCG} G. Besson, G. Courtois, S. Gallot, \emph{Entropies et rigidit\'e des espaces
localement sym\'etriques de courbure strictement n\'egative}, Geometric And
Functional Analysis \textbf{5} (1995), 731--799.

\bibitem[Bi]{Bi} K. Biswas, \emph{Hyperbolic P-barycenters, circumcenters, and Moebius maps},
preprint arXiv 1711.02559.

\bibitem[BuKa]{bk} K. Burns, A. Katok, \emph{Manifolds with non-positive curvature}, Ergodic Theory and Dynamical Systems \textbf{5} (1985) no 2, 307--317.

\bibitem[BuLi]{BuLi} O. Butterley, C. Liverani, \emph{Smooth Anosov flows: correlation spectra and stability.} J. Mod. Dyn. \textbf{1} (2007) no 2, 301--322.

\bibitem[Cr1]{cr1} C. B. Croke, \emph{Rigidity for surfaces of nonpositive curvature}, Comment. Math. Helv. \textbf{65} (1990), no. 1, 150169.
	
\bibitem[Cr2]{Cr2} C.B. Croke, \emph{Rigidity theorems in Riemannian geometry}, Chapter in Geometric Methods in Inverse Problems and PDE Control, C. Croke, I. Lasiecka, G. Uhlmann, and M. Vogelius eds., Springer 2004.

\bibitem[CrDa]{CrDa} C.B. Croke, N. Dairbekov, \emph{Lengths and volumes in Riemannian manifolds}, Duke Math. J. \textbf{125} (2004), no. 1, 1--14.

\bibitem[CDS]{cds} C. B. Croke, N. S. Dairbekov, V. A. Sharafutdinov, \emph{Local boundary rigidity of a compact Riemannian manifold with curvature bounded above}, Trans. Amer. Math. Soc. \textbf{352} (2000), no. 9, 3937--3956.
	
\bibitem[CFF]{cff} C. B. Croke, A. Fathi, J. Feldman, \emph{The marked length spectrum of a surface of non-positive curvature}, Topology \textbf{31} (1992), 847--855.

\bibitem[CrSh]{cs} C. B. Croke, V. A. Sharafutdinov, \emph{Spectral rigidity of a compact negatively curved manifold}, Topology \textbf{37} (1998), no. 6, pp. 1265--1273.	

\bibitem[DaSh]{DaSh} N . Dairbekov, V. Sharafutdinov, \emph{Some problems of integral geometry on Anosov manifolds.}
Erg. Th. Dyn. Sys. \textbf{23} (2003), 59--74.

\bibitem[DGRS]{DGRS} N.V. Dang, C. Guillarmou, G. Rivi\`ere, S. Shen, \emph{Fried conjecture in small dimensions}, preprint.

\bibitem[DMM]{DMM} R. de la Llave, J. M. Marco, R. Moriyon, \emph{Canonical perturbation theory of Anosov systems and regularity results for the Livsic cohomology equation}, Annals of Mathematics \textbf{123} (1986) no. 3, 537--611
	
\bibitem[DyZw]{dz} S. Dyatlov, M. Zworski, \emph{Dynamical zeta functions for Anosov flows via microlocal analysis}, Annales de l'ENS, \textbf{49} (2016), 543--577.
	
\bibitem[DyZw2]{dz2} S. Dyatlov, M. Zworski, \emph{Mathematical theory of scattering resonances}, book in preparation, http://math.mit.edu/$\sim$dyatlov/res/res\_20180406.pdf
	

\bibitem[FaMa]{FaMa} B. Farb, D. Margalit, A primer on mapping class group, Princeton Mathematical Series 49, Princeton Univ. Press.


\bibitem[FaSj]{fs} F. Faure, J. Sj\"ostrand, \emph{Upper bound on the density of Ruelle resonances for Anosov flows}, Comm. Math. Phys. \textbf{308} (2011) 308--325.
	

\bibitem[Gr]{Gr} M. Gromov, \emph{Hyperbolic Groups}, in Essays in Group Theory, MSRI publication vol 8, pages 75-264.


\bibitem[Gu1]{gu} C. Guillarmou, \emph{Invariant distributions and X-ray transform for Anosov flows},
J. Differential Geom. \textbf{105} (2017), no 2, 177--208.
	
\bibitem[Gu2]{gu2} C. Guillarmou, \emph{ Lens rigidity for manifolds with hyperbolic trapped set},
J. Amer. Math. Soc. \textbf{30} (2017), 561--599.

\bibitem[GuKa]{GuKa} V. Guillemin, D. Kazhdan, \emph{Some inverse spectral results for negatively curved 2-manifolds}, Topology \textbf{19} (1980), 301--312.

\bibitem[Ham]{Ham} R. S. Hamilton, \emph{A compactness property for solutions of the Ricci flow}, Amer. Journ. of Math. \textbf{117} (1995) no 3, 545--572.

\bibitem[Ha]{Ha} U. Hamenst\"adt, \emph{Cocycles, symplectic structures and intersection}, Geom. Funct. Anal. \textbf{9} (1999) 90--140
	

\bibitem[Ka]{Ka} A. Katok, \emph{Four applications of conformal equivalence to geometry and dynamics}, Ergodic Theory and Dynamical Systems \textbf{8} (1988), 139--152.
         
         
\bibitem[Kl]{wk} W. Klingenberg, \emph{Riemannian manifolds with geodesic flow of Anosov type,}
		Annals of Mathematics, Second Series, \textbf{99} (1974), no. 1 pp. 1--13
		
\bibitem[Kl2]{Kl2} W. Klingenberg, \emph{Riemannian geometry}, De Gruyter 1982. 

\bibitem[Kn]{Kn} G. Knieper, \emph{New results on noncompact harmonic manifolds},
 Comment. Math. Helv. \textbf{87}, no. 3 (2012), 669-703.	
 
\bibitem[Le]{tl} T. Lefeuvre, \emph{Local marked boundary rigidity under hyperbolic trapping assumptions}, preprint, \url{https://arxiv.org/abs/1804.02143}.
		
%
\bibitem[Li]{anl} A. N. Livsic, \emph{Cohomology of dynamical systems,}
	Izv. Akad. Nauk. SSSR Ser. Mat. Tome \textbf{36} (1972).
	
\bibitem[LoTh]{lt} A. Lopes, Ph. Thieullen, \emph{Sub-actions for Anosov flows}, Ergodic Theory and Dynamical Systems \textbf{25} (2005) no 2, 605 -- 628.
         
\bibitem[Ot]{ot} J-P. Otal, \emph{Le spectre marqu\'e des longueurs des surfaces \`a courbure 
n\'egative}, Ann. of Math. (2) \textbf{131} (1990), no. 1, 151--162
         
\bibitem[Pa]{wp} W. Parry, \emph{Equilibrium states and weighted uniform distribution of closed orbits}, Dynamical Systems, Lecture Notes in Math., Vol. 1342, 626-637, Springer-Verlag, Berlin (1988).
	
         
\bibitem[PSU]{psu1} G. P. Paternain, M. Salo, G. Uhlmann, \emph{Spectral rigidity and invariant distributions on Anosov surfaces},  J. Diff. Geom. \textbf{98} (2014), no. 1, 147--181.
 
\bibitem[PSU2]{psu2} G. P. Paternain, M. Salo, G. Uhlmann, \emph{Invariant distributions, Beurling transforms and tensor tomography in higher dimensions.}
Math. Ann. \textbf{363} (2015), no. 1, 305--362

\bibitem[PeUh]{PeUh}  L. Pestov, G.Uhlmann, \emph{Two dimensional compact simple Riemannian manifolds are boundary distance rigid.}
 Ann. of Math. (2) \textbf{161} (2005), no. 2, 1093--1110.	

    
 \bibitem[PoSh]{ps} M. Pollicott, R. Sharp, \emph{Livsic theorems, maximizing measures and the stable norm}, Dynamical Systems, \textbf{19} (2004), no 1, 75--88,
         
 \bibitem[Sh]{sh} V. A. Sharafutdinov, Integral geometry of tensor fields, 1994.
          
\bibitem[Si]{Si} K. Sigmund, \emph{On the Space of Invariant Measures for Hyperbolic Flows}, Amer. Journ. Math. \textbf{94} (1972), no. 1, 31--37.
    
\bibitem[SUV]{SUV} P. Stefanov, G. Uhlmann, A. Vasy, \emph{Local and global boundary rigidity and the geodesic X-ray transform in the normal gauge}, arXiv:1702.03638.
    
\bibitem[Vi]{mfv} M-F. Vign\'eras, \emph{Vari\'et\'es Riemanniennes isospectrales non-isom\'etriques}, Annals of Mathematics, Second Series, \textbf{112} (1980), no. 1, 21--32.
	
\bibitem[Wi]{aw} A. Wilkinson, \emph{Lectures on marked length spectrum rigidity,}
		Lecture notes, \url{www.math.utah.edu/pcmi12/lecture_notes/wilkinson.pdf}

\end{thebibliography}


\end{document}